\documentclass[11pt]{amsart}

\usepackage[T1]{fontenc}
\usepackage[utf8]{inputenc}
\usepackage{lmodern}
\usepackage{amsmath,amssymb,amsthm,mathtools}

\mathtoolsset{showonlyrefs=true}

\usepackage[margin=1.1in]{geometry}
\usepackage{microtype}
\usepackage{enumitem}
\usepackage[colorlinks=true,
linkcolor=blue,
citecolor=blue,
urlcolor=blue]{hyperref}

\newtheorem{theorem}{Theorem}[section]
\newtheorem{lemma}[theorem]{Lemma}
\newtheorem{proposition}[theorem]{Proposition}

\theoremstyle{definition}
\newtheorem{definition}[theorem]{Definition}
\theoremstyle{remark}
\newtheorem{remark}[theorem]{Remark}

\newcommand{\R}{\mathbb{R}}
\newcommand{\cE}{\mathcal{E}}
\newcommand{\cJ}{\mathcal{J}}
\newcommand{\dd}{\mathrm{d}}
\newcommand{\im}{\operatorname{im}}

\title[Topology and regularity of small sub-Riemannian balls]{Topology and regularity of small sub-Riemannian balls under an index condition}

\author{Luca Rizzi}
\address[Rizzi]{SISSA, Via Bonomea 265, 34136 Trieste (IT)}
\email[Rizzi]{lrizzi@sissa.it}

\author{Giuditta Rovelli}
\address[Rovelli]{Université Paris-Saclay, CNRS, CentraleSupélec, Laboratoire des signaux et systèmes, 91190, Gif-sur-Yvette, France.}
\email[Rovelli]{giuditta.rovelli@centralesupelec.fr}

\thanks{This work was supported by the European Research Council (ERC) under the European Union’s Horizon 2020 research and innovation programme (grant agreement GEOSUB, No. 945655), and by the IRP project GEOSUBMAN by INSMI-CNRS. The authors also acknowledge the INdAM support.}

\begin{document}

\begin{abstract}
We study the topology and regularity of small sub-Riemannian balls under an index condition on length-minimizing curves in the metric tangent space. We prove that sufficiently small spheres are Lipschitz hypersurfaces homeomorphic to Euclidean spheres, while closed balls are homeomorphic to Euclidean closed balls, and open balls are diffeomorphic to Euclidean open balls. These results improve the known ones and extend them beyond two-step distributions. The proof relies on uniform Lipschitz estimates for the distance, of independent interest.
\end{abstract}

\maketitle

\section{Introduction and statement}

In Riemannian geometry, the topology of metric balls is controlled by the exponential map: below the injectivity radius at any given point, the exponential map identifies metric balls with Euclidean ones.  In particular, sufficiently small spheres are diffeomorphic to \(S^{n-1}\), closed balls are diffeomorphic to the closed Euclidean ball, and open balls are diffeomorphic to \(\mathbb R^n\).

In the genuinely sub-Riemannian case, this argument breaks down, since the exponential map is singular at the zero covector. The sub-Riemannian distance induces the manifold topology, but this gives no direct information on the topology of metric spheres and balls. In particular, it is not clear in general whether sufficiently small balls are homeomorphic to Euclidean balls. This specific question was singled out as an open problem by Montgomery in \cite[Section~10.3]{Montgomery}. 

A general result in this direction is due to Baryshnikov
\cite{Baryshnikov}, who proved that sufficiently small sub-Riemannian
spheres and balls for two-step distributions are homeomorphic to their
Euclidean counterparts.  His argument replaces the radial contraction of
the Riemannian picture by a suitable contracting flow and studies the
distance along its orbits.  
The details of the proof are only sketched, and we were not able to reproduce parts of the strategy there. 

Our main result improves \cite{Baryshnikov} proving additional regularity results and, more importantly, replacing the two-step assumption by a second-order index condition on the length-minimizing trajectories of the nilpotent approximation (i.e.\ the metric tangent space). Our assumption covers a substantially larger class of structures than the two-step case. 

We state the main result; for preliminaries see Section~\ref{sec:preliminaries}. Let \(M\) be a connected smooth manifold of dimension \(n\geq1\), let
\(q_0\in M\), and let \(X_1,\ldots,X_m\) be smooth vector fields on \(M\)
satisfying the H\"ormander condition at every point.  Let \(d\) be the
corresponding Carnot--Carath\'eodory distance, and denote by
\(\widehat X_1,\ldots,\widehat X_m\) the nilpotent approximations at \(q_0\)
in privileged coordinates. 
For $\varepsilon>0$, denote the sub-Riemannian spheres, closed balls and open balls by
\begin{equation}
 \begin{aligned}
 S_{q_0}(\varepsilon)
   &:=\{q\in M:d(q_0,q)=\varepsilon\},\\
 \overline{B}_{q_0}(\varepsilon)
   &:=\{q\in M:d(q_0,q)\leq\varepsilon\},\\
 B_{q_0}(\varepsilon)
   &:=\{q\in M:d(q_0,q)<\varepsilon\}.
 \end{aligned}
\end{equation}
We also write $S^{n-1}:=\{z\in \R^n: |z|=1\}$, $B^n:=\{z\in\R^n:|z|<1\}$ and
$\overline B^n:=\{z\in\R^n:|z|\leq1\}$.

We use the notions of Morse index and corank of an abnormal trajectory
in the sense of Agrachev--Sarychev \cite{ASmorse}; precise definitions
are recalled in Section~\ref{sec:endpoint-preliminaries}.

\begin{theorem}\label{thm:main}
Assume that every nonconstant abnormal length-minimizing trajectory of the
nilpotent approximation $(\widehat X_1,\ldots,\widehat X_m)$ starting at $0$
has Morse index strictly larger than its corank.
Then there exists $\varepsilon_*>0$ such that, for every
$0<\varepsilon\leq\varepsilon_*$,
there is a homeomorphism of pairs
\begin{equation}
 \bigl(\overline B^n,S^{n-1}\bigr)
 \simeq
 \bigl(\overline{B}_{q_0}(\varepsilon),S_{q_0}(\varepsilon)\bigr).
\end{equation}
Moreover, the open metric ball $B_{q_0}(\varepsilon)$ is diffeomorphic to $B^n$.
\end{theorem}

As a consequence of the proof of Theorem \ref{thm:main} we also have the following regularity result. Here Lipschitz regularity is understood with respect to smooth local
coordinates.

\begin{theorem}
\label{thm:lipschitz-spheres}
In the setting of Theorem~\ref{thm:main}, for all $0<\varepsilon\leq \varepsilon_*$, the metric sphere $S_{q_0}(\varepsilon)$ is an embedded Lipschitz hypersurface.
Moreover, in rescaled privileged coordinates, these spheres admit
uniformly bi-Lipschitz parametrizations by $S^{n-1}$.
\end{theorem}

\begin{remark}
    In the three-dimensional Heisenberg group, metric spheres have conical
singularities at the two points on the vertical axis; see \cite{ABB}. Thus the conclusion of
Theorem~\ref{thm:main} for spheres cannot be strengthened to a
diffeomorphism of embedded hypersurfaces.
\end{remark}
\begin{remark}[A sufficient condition]\label{rmk:sufficient}
A sufficient condition for the hypothesis of Theorem~\ref{thm:main} is that
no nonconstant minimizing trajectory of the nilpotent approximation admit
a Goh--Legendre abnormal lift (Definition~\ref{def:goh-legendre}).
Indeed, finite negative Morse index for an abnormal multiplier along a
minimizer forces the lift of that same multiplier to satisfy the Goh
and generalized Legendre conditions; see \cite[Theorem~12.12]{ABB}.
In their absence, the Morse index of any abnormal length-minimizer is infinite, and our assumption is clearly verified.
\end{remark}
\begin{remark}[On the index condition]
    The index condition in Theorem~\ref{thm:main} is closely related to the second-order criteria developed by Agrachev and Sarychev \cite{ASmorse,ASsubanalytic} used to establish subanalyticity of sub-Riemannian distances \cite{AgrachevCompactness,AgrachevGauthier}. In fact, along length-minimizers, our condition is equivalent to the nonnegativity of the index used in those works. A discrepancy arises from their use of the endpoint map restricted to the sphere of controls, rather than the unrestricted endpoint map as in our formalism. Thus, in the real-analytic setting, our hypothesis also implies subanalyticity of the germ of the sub-Riemannian distance at the base point; see \cite{AgrachevGauthier}.

\end{remark}
The key technical result in the proof of Theorem~\ref{thm:main} is an Euclidean Lipschitz estimate for the family of approximating distances $d_s$, defined  via blowup in Section \ref{sec:blowup}. More precisely, Proposition~\ref{prop:uniform-stability}, gives uniform
Euclidean Lipschitz estimates for the family $d_s(0,\cdot)$ away from the origin, together with Lipschitz dependence on $s$. Proposition~\ref{prop:two-point-lipschitz}
extends these estimates to compact sets of pairs of points away from
the diagonal. The proof of these facts, that have independent interest, uses nonsmooth analysis, a viewpoint inspired by \cite{RiffordSard}.

For a fixed sub-Riemannian structure, Lipschitz regularity of the distance was established in
\cite[Theorem~5.7]{AgrachevLee} and \cite[Theorem~12.11]{ABB}. Compare also with \cite[Theorem~3.15]{Rifford}.


As a further consequence of the proof of Theorem \ref{thm:main}, all sufficiently small metric
spheres are embedded Lipschitz hypersurfaces
(Theorem~\ref{thm:lipschitz-spheres}). For homogeneous distances, the implication from local Lipschitz regularity
of the distance to that of its spheres was established in
\cite[Lemma~4.3]{LeDonneNicolussiGolo}. In the general setting,
Rifford \cite[Theorem~3.1]{RiffordSpheres} proved Lipschitz regularity of
spheres for almost every radius on complete sub-Riemannian manifolds,
assuming that no nonconstant minimizing curve from the center is singular. See also \cite{LeDonneNalon} for Euclidean rectifiability of sub-Finsler
spheres in free step-two Carnot groups. Finally, the existence of tangent hyperplanes to spheres is discussed in \cite{Agrachev-hyperplanes,Trelat-hyperplanes}. 

\subsection{Legendre-indefinite families}
\label{sec:legendre-indefinite}

The sufficient condition of Remark \ref{rmk:sufficient} holds for a step-two structure, since it is well-known that there cannot be nonconstant Goh abnormal minimizers.  It also holds, trivially, when the nilpotent approximation has no
nonconstant abnormal minimizers. We stress that the hypothesis of Theorem~\ref{thm:main} is strictly weaker than the absence of Goh
minimizers: the rank-three, step-three Carnot group
$N_{6,3,1}$ (in the notation of \cite{cornucopia}) has nonconstant Goh
minimizing curves but no nonconstant trajectories admitting a
Goh--Legendre abnormal lift; see \cite[proof of Theorem~7.6 and Remark~7.7]{BorzaRizzi}.

A sufficient condition can be formulated in terms of a generating family
$X_1,\ldots,X_m$ (see Section~\ref{sec:preliminaries}).  For
$u\in\R^m$, write $X_u:=\sum_i u_iX_i$.
For $x\in M$ and $u\in\R^m$ such that $X_u(x)\neq0$, set
\begin{equation}
 E_x(u):=\operatorname{span}\bigl\{
 X_v(x),\ [X_v,X_w](x),\ [X_u,[X_v,X_w]](x):
 v,w\in\R^m\bigr\}.
\end{equation}
For $0\neq\lambda\in E_x(u)^\perp$, define the quadratic form
$Q_{x,u,\lambda}:\R^m\to\R$ by
\begin{equation*}
 Q_{x,u,\lambda}(v)
 :=\bigl\langle\lambda,
       [[X_u,X_v],X_v](x)\bigr\rangle.
\end{equation*}

\begin{definition}[Legendre-indefinite families]
\label{def:legendre-indefinite}
A family of vector fields $X_1,\ldots,X_m$ is \emph{Legendre-indefinite} if, for every
$x\in M$, every $u\in\R^m$ with $X_u(x)\neq0$, and every
$0\neq\lambda\in E_x(u)^\perp$, the quadratic form
$Q_{x,u,\lambda}$ is indefinite.  Equivalently, there are
$v_+,v_-\in\R^m$ such that
\begin{equation}
 Q_{x,u,\lambda}(v_+)>0,
 \qquad
 Q_{x,u,\lambda}(v_-)<0.
\end{equation}
\end{definition}
\begin{remark}\label{rmk:invarianceassertion}
Definition \ref{def:legendre-indefinite} is invariant under smooth invertible changes of generators,
with the corresponding change of control. 
These facts are proved in Lemma~\ref{lem:legendre-indefinite} below.
\end{remark}
\begin{remark}
    If $\Delta\subset TM$ is a distribution, the above objects can be defined in a similar way in terms of a local frame for $\Delta$. In this case, $E_x(u)$ depends only on $X=X_u(x)$ and it holds
\begin{equation}
E_x(X)= \Delta(x)+[\Delta,\Delta](x)+[X,[\Delta,\Delta]](x),
\end{equation}
where the brackets are evaluated using local horizontal extensions. Furthermore, $Q_{x,u,\lambda}(v)$ depends only on $X$ and $V=X_v(x)$. So, both objects do not depend on the choice of local frames.
We define Legendre-indefinite distributions as those for which local frames are Legendre-indefinite. 
\end{remark}

By
Lemma~\ref{lem:legendre-indefinite}, if the family
$\widehat X_1,\ldots,\widehat X_m$ of the nilpotent approximation is
Legendre-indefinite, no nonconstant trajectory admits a Goh--Legendre
 lift. Thus every nonconstant abnormal length-minimizing trajectory
has infinite Morse index and satisfies the hypothesis of
Theorem~\ref{thm:main}.

Legendre-indefiniteness extends the medium-fat condition introduced by
Agrachev and Sarychev
\cite{ASsubanalytic}; see also
\cite[Example~2.1]{Rifford}.  Indeed, for medium-fat distributions
$E_x(X)=T_xM$ whenever $X(x)\neq0$, so the condition in
Definition~\ref{def:legendre-indefinite} is vacuous.  Every medium-fat
distribution is therefore Legendre-indefinite. The converse does not hold, in general: the Carnot group $N_{6,3,1}$ mentioned above has a Legendre-indefinite distribution, that is not medium-fat.

\subsection{Closing thoughts}
We do not know whether the conclusions of Theorem~\ref{thm:main} hold without assumptions on the nilpotent approximation. At a more elementary level, we do not even know whether small sub-Riemannian balls are contractible, or whether
\begin{equation}
 B_{q_0}(r) = \operatorname{int}\overline B_{q_0}(r).
\end{equation}
In particular, it is open whether strict inclusion can occur for arbitrarily small radii.

\subsection{Acknowledgements} We thank A. Agrachev for suggesting the possibility of extending the result of \cite{Baryshnikov} to structures with no nonconstant Goh-abnormal trajectories.

\subsection{AI disclosure} Most of the mathematical content of this paper was developed in the context of Giuditta Rovelli's Master's thesis at the University of Trieste \cite{RovelliMasterThesis}. AI tools were used only for manuscript preparation and human-guided improvements to exposition. All mathematical results and proofs were developed and verified by humans.

\section{Preliminaries}\label{sec:preliminaries}

We collect the notation and standard facts used in the proof.  Further
background can be found in \cite{ABB,Bellaiche,ASbook}.

\subsection{Control system and distance}

Let $X_1,\ldots,X_m$ be smooth vector fields on a connected manifold $M$. Their rank may be nonconstant on $M$.
For each integer $r\geq1$, define
\begin{equation}
 \mathcal L^r(q):=\operatorname{span}\bigl\{
 [X_{i_1},[\cdots,[X_{i_{k-1}},X_{i_k}]\cdots]](q):
 1\leq k\leq r,\ 1\leq i_1,\ldots,i_k\leq m\bigr\},
\end{equation}
where the expression for $k=1$ means $X_{i_1}(q)$.
We assume the \emph{bracket-generating}, or \emph{H\"ormander}, condition to hold. Namely
for every $q\in M$ there is an integer $r(q)\geq1$ such that
\begin{equation}
 \mathcal L^{r(q)}(q)=T_qM.
\end{equation}
An absolutely continuous curve $\gamma:[0,1]\to M$ is
\emph{horizontal}, or \emph{admissible}, if there exists
$u\in L^1([0,1];\R^m)$ such that
\begin{equation}\label{eq:control-system}
 \dot\gamma(t)=\sum_{i=1}^m u_i(t) X_i(\gamma(t))
 \quad\text{for almost every }t\in[0,1].
\end{equation}
 For a given $u\in L^1([0,1];\R^m)$, we denote by $\gamma_u$ the
solution of \eqref{eq:control-system} with a fixed initial point,
provided it exists on $[0,1]$. For $q\in M$ and $\xi\in\operatorname{span}\{X_1(q),\ldots,X_m(q)\}$,
set
\begin{equation}
 |\xi|_q:=\min\{|v|:X_v(q)=\xi\}.
\end{equation}
Therefore, the length of a horizontal curve is
\begin{equation*}
 \ell(\gamma):=\int_0^1|\dot\gamma(t)|_{\gamma(t)}\,\dd t.
\end{equation*}
Every horizontal curve has a unique associated \emph{least-norm control}
$u_\gamma\in L^1([0,1];\R^m)$. The latter is characterized by the following equality for almost every $t$:
\begin{equation}
 X_{u_\gamma(t)}(\gamma(t))=\dot\gamma(t),\qquad
 |u_\gamma(t)|=|\dot\gamma(t)|_{\gamma(t)}.
\end{equation}
In particular, $\ell(\gamma)=\|u_\gamma\|_{L^1}$.
A horizontal curve has \emph{constant speed} if
$|\dot\gamma(t)|_{\gamma(t)}$, equivalently $|u_\gamma(t)|$, is constant
almost everywhere.

The associated sub-Riemannian, or Carnot--Carath\'eodory, distance is defined as
\begin{equation*}
 d(p,q):=\inf\{\ell(\gamma):\gamma(0)=p,\ \gamma(1)=q,\
                              \gamma\text{ horizontal}\}.
\end{equation*}
The bracket-generating hypothesis and the Chow--Rashevskii theorem imply
that $d$ is finite and induces the manifold topology; see, for example,
\cite[Chapter~3]{ABB}. Curves realizing this infimum are called
\emph{length-minimizing}.

\subsection{Energy formulation and endpoint maps} 
It is convenient to introduce the following energy minimization problem. Set $H:=L^2([0,1];\R^m)$ and
$\cJ(u):=\tfrac12\|u\|_H^2$. For $u\in H$ whose horizontal trajectory is defined, it holds
\begin{equation}\label{eq:CS}
     \ell(\gamma_u)^2=\|u_{\gamma_u}\|_{L^1}^2
 \leq\|u_{\gamma_u}\|_{L^2}^2
 \leq\|u\|_{L^2}^2=2\cJ(u).
\end{equation}
The first inequality is Cauchy--Schwarz on $[0,1]$; the second follows
from minimality. Equality throughout holds if and only if $\gamma_u$ has constant speed and $u$ is its least-norm control.

For fixed initial point $p\in M$, we define  the \emph{endpoint
map} $\cE : \mathcal U \to M$ by $\cE(u):=\gamma_u(1)$, where $\mathcal{U}\subset H$ is the open set of controls where the solution of \eqref{eq:control-system} with $\gamma_u(0)=p$ exists on $[0,1]$. Then
\begin{equation}\label{eq:energy-distance-general}
 \tfrac12d(p,q)^2=\inf\{\cJ(u):u\in\mathcal U,\ \cE(u)=q\}.
\end{equation}
One direction follows from \eqref{eq:CS}. The reverse direction follows
by parametrizing a sequence of curves with lengths tending to $d(p,q)$
at constant speed and taking their least-norm controls.
The controls attaining this infimum, called \emph{energy-minimizing},
are precisely the least-norm controls of constant-speed length-minimizing
curves. In particular, they belong to $L^\infty([0,1];\R^m)$.

\subsection{Subdifferentials and Lipschitz functions}
\label{sec:subdifferentials}

Let $O\subset\R^n$ be open and $f:O\to\R$ continuous. A
\emph{support function from below} for $f$ at $x\in O$ is a function
$\varphi$ defined on a neighborhood of $x$ such that $\varphi\leq f$ and $\varphi(x)=f(x)$.
The \emph{Fr\'echet subdifferential} is
\begin{equation}
 \partial^-f(x):=\{\dd_x\varphi:\varphi\text{ is a }C^1
             \text{ support function from below at }x\}.
\end{equation}
The \emph{proximal subdifferential} $\partial^-_P f(x)$ is defined in
the same way using $C^2$ support functions. In particular,
$\partial^-_P f(x)\subseteq\partial^-f(x)$.
We use the Euclidean dual norm for covectors.

The following are standard facts of nonsmooth analysis; see, e.g.,
\cite{CLSW,RockafellarWets}. We include short proofs in
Appendix~\ref{app:nonsmooth}; compare
also \cite[Propositions~2.2 and~2.4]{RiffordSard}, respectively.

\begin{lemma}[Approximation by proximal subdifferentials]
\label{lem:proximal-approximation}
For every $x\in O$, $p\in\partial^-f(x)$, and neighborhood $W$ of
$(x,p)$ in $O\times(\R^n)^*$, there are $y\in O$ and
$q\in\partial^-_P f(y)$ such that $(y,q)\in W$.
\end{lemma}

\begin{proposition}[Bounded subdifferentials and Lipschitz continuity]
\label{prop:subdifferential-lipschitz}
Let $C\subset O$ be open and convex, and let $L\geq0$. Then $f$ is
$L$-Lipschitz on $C$ if and only if
\begin{equation*}
 |p|\leq L\qquad\text{for every }x\in C
                     \text{ and every }p\in\partial^-f(x).
\end{equation*}
\end{proposition}

\subsection{Morse index, corank, and Goh--Legendre lifts}
\label{sec:endpoint-preliminaries}

The endpoint map $\cE$ is smooth.
For $0\leq t\leq\tau\leq1$ and $u \in \mathcal U$, let $\Phi^u_{t,\tau}$ be the flow from time $t$
to time $\tau$ of $X_{u(\cdot)}:=\sum_i u_i(\cdot)X_i$, defined near $\gamma_u(t)$.
The differential of $\cE$ is given by the well-known formula
\begin{equation}\label{eq:endpoint-differential}
 \dd_u\cE(v)
 =\int_0^1 \dd_{\gamma_u(t)}\Phi^u_{t,1}
    \left(\sum_{i=1}^m v_i(t)X_i(\gamma_u(t))\right)\dd t.
\end{equation}

A control is \emph{singular}, or \emph{abnormal}, if it is a critical point
of $\cE$, that is, if $\dd_u\cE$ is not surjective. Any nonzero $\lambda_1 \in(\im\dd_u\cE)^\perp$ is called an \emph{abnormal (Lagrange) multiplier}. 

Following \cite[Section~12.1]{ABB}, the intrinsic Hessian at a critical
control $u$ is the quadratic map
\begin{equation}
 \operatorname{Hess}_u\cE:\ker\dd_u\cE
    \longrightarrow T_{\gamma_u(1)}M/\im\dd_u\cE,\qquad
 v\longmapsto \dd_u^2\cE[v,v].
\end{equation}
The second differential is computed in coordinates, but its class in the
quotient is independent of this choice. For an abnormal multiplier
$\lambda_1\in(\im\dd_u\cE)^\perp$, we associate the quadratic form
\begin{equation}
 q_{\lambda_1}(v):=
 \lambda_1\bigl(\operatorname{Hess}_u\cE(v)\bigr),
 \qquad v\in\ker\dd_u\cE.
\end{equation}
The \emph{negative Morse index}
$\operatorname{ind}_-q_{\lambda_1}$ is the supremum of the dimensions of
subspaces of $\ker\dd_u\cE\subset H$ on which $q_{\lambda_1}$ is negative
definite.

Length-minimizing trajectories are henceforth parametrized at constant
speed on $[0,1]$ and represented by their unique least-norm controls.
A trajectory $\gamma_u$ is called
\emph{abnormal} when its least-norm control $u$ is singular.
Following \cite{ASmorse}, its \emph{corank} and
\emph{Morse index} are
\begin{equation*}
 \operatorname{corank}(\gamma_u):=\operatorname{codim}\im\dd_u\cE,
 \qquad
 \operatorname{ind}(\gamma_u):=
 \min_{0\neq\lambda_1\in(\im\dd_u\cE)^\perp}
       \operatorname{ind}_-q_{\lambda_1}.
\end{equation*}
The minimum is well defined in $\mathbb N\cup\{0,+\infty\}$. Writing $k=\operatorname{corank}(\gamma_u)$, we have
\begin{equation*}
 \operatorname{ind}(\gamma_u)>\operatorname{corank}(\gamma_u)
 \quad\Longleftrightarrow\quad
 \operatorname{ind}_-q_{\lambda_1} > k
 \quad\text{for every }0\neq\lambda_1\in(\im\dd_u\cE)^\perp.
\end{equation*}
For regular trajectories there are no abnormal multipliers,
so the condition is vacuous.

\begin{definition}[Goh--Legendre abnormal lift]
\label{def:goh-legendre}
A control $u$ and its trajectory $\gamma_u$ admit a
\emph{Goh--Legendre abnormal lift} if there is a covector
$\lambda_1\in T^*_{\gamma_u(1)}M\setminus\{0\}$ such that the lift
\begin{equation}\label{eq:lift}
     \lambda(t):=(\Phi^u_{t,1})^*\lambda_1
 \in T^*_{\gamma_u(t)}M
\end{equation}
satisfies the following conditions for almost every $t\in[0,1]$:
\begin{align}
 \langle\lambda(t),X_v(\gamma_u(t))\rangle&=0,
   &&v\in\R^m,\label{eq:abnormal} \\
 \langle\lambda(t),[X_v,X_w](\gamma_u(t))\rangle&=0,
   &&v,w\in\R^m,\label{eq:goh}\\
 \bigl\langle\lambda(t),
 [[X_{u(t)},X_v],X_v](\gamma_u(t))\bigr\rangle&\geq0,
   &&v\in\R^m.\label{eq:legendre}
\end{align}
Equation \eqref{eq:goh} is the \emph{Goh condition}, while \eqref{eq:legendre} is the \emph{generalized
Legendre condition}.
\end{definition}

By \eqref{eq:endpoint-differential}, condition \eqref{eq:abnormal} is equivalent to
$\lambda_1\in(\im\dd_u\cE)^\perp$. A key fact is the following: if $u$ is energy-minimizing and $\operatorname{ind}_-q_{\lambda_1}<\infty$
for an abnormal multiplier $\lambda_1$, the lift \eqref{eq:lift} of that multiplier
satisfies the Goh and generalized Legendre conditions. See \cite[Theorem~12.12]{ABB} or \cite[Proposition~10.1]{ASmorse}.
Consequently, if an abnormal minimizing trajectory admits no such lift,
then $\operatorname{ind}_-q_{\lambda_1}=+\infty$ for every abnormal
multiplier $\lambda_1$, so its Morse index is $+\infty$.

\begin{lemma}
\label{lem:legendre-indefinite}
A Legendre-indefinite family (see Definition \ref{def:legendre-indefinite}) has no nonconstant horizontal trajectories
with a Goh--Legendre abnormal lift. Hence every nonconstant abnormal
length-minimizing trajectory has infinite Morse index.
\end{lemma}

\begin{proof}
First, we prove the invariance assertions in
Section~\ref{sec:legendre-indefinite} for $E_x(u)$ and $Q_{x,u,\lambda}$, see Remark \ref{rmk:invarianceassertion}. To this purpose, recall that
\begin{equation}
 \mathcal L^2(x)=\operatorname{span}\{X_i(x),[X_i,X_j](x):1\leq i,j\leq m\}.
\end{equation}
If $a,b,c$ are smooth $\R^m$-valued functions near $x$, with
$a(x)=u$, $b(x)=v$, and $c(x)=w$, the Leibniz rule gives
\begin{align*}
 [X_b,X_c](x)& =[X_v,X_w](x)
       &&\pmod{\operatorname{span}\{X_i(x)\}},\\
 [X_a,[X_b,X_c]](x)& = [X_u,[X_v,X_w]](x)
       &&\pmod{\mathcal L^2(x)},\\
 [[X_a,X_b],X_b](x)& = [[X_u,X_v],X_v](x)
       &&\pmod{\mathcal L^2(x)}.
\end{align*}
Here $X_a:=\sum_i a_iX_i$, and similarly for $X_b,X_c$.
In particular, replacing $X_i$ by
$\widetilde X_j:=\sum_i A_{ij}X_i$, where $A$ is a smooth invertible
matrix, gives
\begin{equation}
 E_x^{\widetilde X}(u)=E_x^X(A(x)u),\qquad
 Q^{\widetilde X}_{x,u,\lambda}(v)
   =Q^X_{x,A(x)u,\lambda}(A(x)v).
\end{equation}
This proves invariance, since $A(x)$ is invertible and
$\widetilde X_u(x)=X_{A(x)u}(x)$.  
If the generators have constant rank, choose a local frame of their
image distribution $\Delta$ and express them in that frame.  The same
Leibniz-rule calculations show that $E_x(u)$ and $Q_{x,u,\lambda}(v)$
depend only on $X_u(x)$ and $X_v(x)$, and give the stated expressions
on $\Delta_x$. 

We proceed now with the proof of the Lemma. By contradiction, suppose that $\gamma=\gamma_u$ admits a Goh--Legendre lift
$\lambda$. The abnormal and Goh conditions give, for almost every $t$,
\begin{equation}\label{eq:goh-annihilates-square}
 \lambda(t)\bigl(\mathcal L^2(\gamma(t))\bigr)=0.
\end{equation}
For each $i,j$, the absolutely continuous function
$t\mapsto\langle\lambda(t),[X_i,X_j](\gamma(t))\rangle$ vanishes
identically. Differentiating gives, for almost every $t$,
\begin{equation}
 0=\frac{\dd}{\dd t}\langle\lambda(t),[X_i,X_j](\gamma(t))\rangle
  =\langle\lambda(t),[X_{u(t)},[X_i,X_j]](\gamma(t))\rangle.
\end{equation}
Since there are finitely many pairs $i,j$, these identities hold on a
common set of full measure. Together with \eqref{eq:goh-annihilates-square},
they imply
\begin{equation*}
 \lambda(t)\in E_{\gamma(t)}(u(t))^\perp
\end{equation*}
at almost every time with $X_{u(t)}(\gamma(t))\neq0$.
The covector $\lambda(t)$ is nonzero by invertibility of the linearized
flow, and \eqref{eq:legendre} gives
$Q_{\gamma(t),u(t),\lambda(t)}\geq0$, contradicting
Legendre-indefiniteness. Thus $\dot\gamma=0$ almost everywhere and
$\gamma$ is constant. 
\end{proof}

\subsection{Privileged coordinates and nilpotent approximation}
\label{sec:privileged-preliminaries}

We follow Bella\"iche \cite{Bellaiche}. Fix $q_0\in M$,
set $n_r:=\dim\mathcal L^r(q_0)$ and $n_0:=0$, and define $w_j=r$
when $n_{r-1}<j\leq n_r$. Privileged coordinates are coordinates
$x=(x_1,\ldots,x_n)$ for which the function $x_j$ has nonholonomic order
$w_j$ at $q_0$.  Equivalently, every iterated derivative of $x_j$ with
respect to the fields $X_1,\ldots,X_m$ of length strictly less than $w_j$ vanishes at
$q_0$, while at least one of length $w_j$ is nonzero.

Choose such coordinates on a neighborhood of $\overline U$, where
$U$ is a relatively compact neighborhood of $q_0$.  Write
\begin{equation}
 x=(x_1,\ldots,x_n):U\longrightarrow x(U)\subset\R^n,
 \qquad x(q_0)=0,
\end{equation}
and define the anisotropic dilation map $\delta_r:\R^n\to\R^n$ by
\begin{equation*}
 \delta_r(x_1,\ldots,x_n)
  =(r^{w_1}x_1,\ldots,r^{w_n}x_n),\qquad r\geq 0.
\end{equation*}
Note that $\delta_0z =0$.  Unless otherwise specified, norms in the coordinate
space are Euclidean.

\begin{lemma}[Localization]\label{lem:localization}
There exists $r_{\mathrm{loc}}>0$ such that, for every
$q,q'\in \overline B_{q_0}(r_{\mathrm{loc}})$,
\begin{equation}
 d(q,q')
 =
 \inf\{\ell(\gamma):
       \gamma(0)=q,\ \gamma(1)=q',\
       \gamma([0,1])\subset U\}.
\end{equation}
\end{lemma}

\begin{proof}
It is sufficient to prove the $\geq$ inequality, since the other direction is immediate. Choose $r_{\mathrm{loc}}>0$ so that
$B_{q_0}(4r_{\mathrm{loc}})\subset U$. For
$q,q'\in\overline B_{q_0}(r_{\mathrm{loc}})$ and
$0<\eta<r_{\mathrm{loc}}$, choose a horizontal curve $\gamma$ from $q$
to $q'$ with $\ell(\gamma)\leq d(q,q')+\eta<3r_{\mathrm{loc}}$. Then
\begin{equation}
 d(q,\gamma(t))\leq\ell(\gamma|_{[0,t]})\leq\ell(\gamma),
\end{equation}
so $d(q_0,\gamma(t))<4r_{\mathrm{loc}}$ and $\gamma$ lies in $U$.
Letting $\eta\downarrow0$ proves the inequality.
\end{proof}

In particular, balls and spheres centered at $q_0$ of radius less than
$r_{\mathrm{loc}}$  lie in $U$.  In coordinates they lie in the relatively
compact neighborhood $x(U)$.

When working in coordinates, we identify each field $X_i$ with its
push-forward $x_*X_i$.  Each $X_i$ has a weighted Taylor expansion.
Its homogeneous component of weighted degree $-1$, denoted by
$\widehat X_i$, is a polynomial vector field defined on all of $\R^n$. In particular, $(\delta_{1/s})_*\widehat X_i=s^{-1}\widehat X_i$. The family
$\widehat X_1,\ldots,\widehat X_m$ is the \emph{nilpotent approximation} at
$q_0$; the Lie algebra generated by these homogeneous fields is nilpotent.
 The corresponding
Carnot--Carath\'eodory distance $\widehat d$ satisfies
\begin{equation}
 \widehat d(0,\delta_rz)=r\widehat d(0,z),\qquad z\in \R^n,\, r\geq 0.
\end{equation}
The sub-Riemannian structure on $\R^n$ defined by the family $\widehat X_1,\ldots,\widehat X_m$ also takes the name of nilpotent approximation, and it is isometric to the metric tangent space at $q_0$. It is a Carnot group at regular points and a Carnot homogeneous space in general, see \cite{Bellaiche} for details.

\subsection{Approximating families}\label{sec:blowup}

For $s>0$ put
\begin{equation}
 U_s:=\delta_{1/s}x(U),
\end{equation}
and set $U_0:=\R^n$.  Note that, for every compact set $K\subset\R^n$, there exists
$s_K>0$ such that $K\subset U_s$ whenever
$0<s\leq s_K$.  On $U_s$, define the rescaled vector fields
\begin{equation}\label{eq:scaled-fields}
 X_i^s:=s\,(\delta_{1/s})_*X_i,
 \qquad i=1,\ldots,m,
\end{equation}
and set $X_i^0:=\widehat X_i$.  Let $d_s$ be the
Carnot--Carath\'eodory distance defined by
$X_1^s,\ldots,X_m^s$ on $U_s$.

\begin{lemma}\label{lem:blowup}
The following assertions hold.
\begin{enumerate}[label=\textup{(\roman*)}]
\item The maps $s\mapsto X_i^s$ extend smoothly to
  $s=0$ in the $C^\infty$ topology on compact subsets, with value
  $X_i^0=\widehat X_i$.

\item The distances $d_s$ converge to $d_0=\widehat d$ uniformly
  on compact subsets of $\R^n\times\R^n$ as
  $s\downarrow0$.

\item For $s>0$ and $z,z'\in U_s$ such that $s d_s(0,z),    s d_s(0,z')\leq r_{\mathrm{loc}}$, one has
  \begin{equation}\label{eq:distance-blowup}
    d_s(z,z') =  \frac{1}{s} d\bigl(x^{-1}(\delta_s z),x^{-1}(\delta_s z')\bigr).
  \end{equation}

\item For $s\geq0$, $r>0$, and $z,z'\in U_{sr}$,
  \begin{equation}\label{eq:two-parameter-scaling}
    d_s(\delta_r z,\delta_r z')
       =r\,d_{sr}(z,z').
  \end{equation}
\end{enumerate}
\end{lemma}

\begin{proof}
We sketch the proof, see \cite{Bellaiche} or \cite[Chapter 10]{ABB} for details. In privileged coordinates, the coefficients of $X_i$ have weighted Taylor
expansions, and in particular formula \eqref{eq:scaled-fields} does not change the terms of weighted
degree $-1$ and multiplies the higher-order terms by nonnegative integral
powers of $s$. More precisely, for every compact set $K\subset\R^n$ there exists
$s_K>0$ such that
\begin{equation}
  X_i^s(z)=\widehat X_i(z)+sW_i^s(z),
  \qquad 0\leq s\leq s_K,\quad z\in K,
\end{equation}
where the $W_i^s$ are smooth vector fields on a neighborhood of $K$
and depend smoothly on $s$. This proves (i).  Assertion (ii) is the standard local uniform
convergence theorem for the nilpotent approximation; see
\cite[Theorem 7.32]{Bellaiche} or \cite[Theorem 10.65]{ABB}.
For a general convergence theorem for distances associated to varying
fields, see \cite[Theorem~1.4]{AntonelliLeDonneNicolussiGolo}.

We prove (iii). Fix $s>0$ and $z,z'\in U_s$, and set 
\begin{equation}
 q:=x^{-1}(\delta_s z),
 \qquad
 q':=x^{-1}(\delta_s z').
\end{equation}
We omit the coordinate map from the notation, for simplicity. Let $\ell_s$ denote the sub-Riemannian length associated with the family $(X_i^s)$.
A horizontal curve $\gamma$ contained in $U$ is horizontal for the original family of vector fields, with control $u$, if and only if the trajectory $\widetilde\gamma:=\delta_{1/s}\circ \gamma$ is horizontal for the family $(X_i^s)$, with control $u/s$. This bijection of horizontal curves preserves least-norm controls
up to the factor $1/s$, hence
\begin{equation}
 \ell_s(\delta_{1/s}\circ\gamma)=s^{-1}\ell(\gamma).
\end{equation}
Taking infima over curves gives
\begin{equation}\label{eq:local-distance-rescaling}
 s d_s(z,z')
 =\inf\{\ell(\gamma):\gamma(0)=q,\ \gamma(1)=q',\
                         \gamma([0,1])\subset U\}.
\end{equation}

This implies, in particular, that $d(q_0,q)\leq s d_s(0,z)\leq r_{\mathrm{loc}}$, and similarly $d(q_0,q')\leq s d_s(0,z')\leq r_{\mathrm{loc}}$. Thus $q,q'\in\overline B_{q_0}(r_{\mathrm{loc}})$, and
Lemma~\ref{lem:localization}, together with
\eqref{eq:local-distance-rescaling}, yields
\eqref{eq:distance-blowup}.

Finally, for $s>0$ a direct computation from
\eqref{eq:scaled-fields} gives
\begin{equation}
  X_i^{sr}=r\,(\delta_{1/r})_*X_i^s.
\end{equation}
Since $\delta_rU_{sr}=U_s$, the same length-rescaling argument gives
$\ell_s(\delta_r\circ\gamma)=r\ell_{sr}(\gamma)$.
Taking infima proves (iv). For $s=0$, the identity follows from
the homogeneity of $\widehat X_i$.
\end{proof}

\section{Uniform Lipschitz estimates for the approximating distances}
\label{sec:lipschitz}

We prove uniform Lipschitz estimates by bounding the proximal
subdifferentials of the squared distances, using the second variation of
the endpoint maps along minimizing controls.
In the setting of Section~\ref{sec:blowup}, for $s\geq0$, let $\mathcal U_s\subset H$ be the open set of controls
whose $X^s$-trajectory from $0$ is defined in $U_s$ on $[0,1]$. Furthermore, set
\begin{equation}
 \cE_s:\mathcal U_s\longrightarrow U_s,\qquad
 \cE_s(u)=\gamma_u^s(1),\qquad
 \dot\gamma_u^s=\sum_{i=1}^m u_iX_i^s(\gamma_u^s),\quad
 \gamma_u^s(0)=0.
\end{equation}
For $z\in U_s$, \eqref{eq:energy-distance-general} gives
\begin{equation}
 f_s(z):=\tfrac12d_s(0,z)^2
        =\inf\{\cJ(u):u\in\mathcal U_s,\ \cE_s(u)=z\}.
\end{equation}

By Section~\ref{sec:endpoint-preliminaries}, the hypothesis of
Theorem~\ref{thm:main} is equivalent to requiring
$\operatorname{ind}_-(\lambda_1\circ\operatorname{Hess}_u\cE_0) > k$,
where $k=\operatorname{codim}\im\dd_u\cE_0$, for every energy-minimizing
control $u$ with $\cE_0(u)\neq0$ and every abnormal multiplier
$\lambda_1$.

The proof of the uniform Lipschitz estimates below requires uniform control of the trajectories and endpoint maps as the approximating system varies with \(s\). The following lemma collects the corresponding estimates.
\begin{lemma}\label{lem:compact-containment}
For every $A>0$, there exist $s_A>0$ and a compact set
$\Omega_A\subset\R^n$, with $\Omega_A\Subset U_s$ for
$0\leq s\leq s_A$, such that every control
$u\in H$ with $\|u\|_H\leq A$ has an
$X^s$-trajectory from $0$ defined on $[0,1]$ and contained in $\Omega_A$.
Moreover, there is $C_A>0$ such that
\begin{equation}\label{eq:ode-parameter-bound}
 \sup_{\tau\in[0,1]}|\gamma_u^s(\tau)-\gamma_u^t(\tau)|
 \leq C_A|s-t|,
 \qquad s,t\in[0,s_A],\quad \|u\|_H\leq A.
\end{equation}
Furthermore, the endpoint maps $\cE_s$, $s\in[0,s_A]$,
are defined on a common open neighborhood in $H$ of the closed ball
$\|u\|_H\leq A$ and satisfy
\begin{equation}
 \sup_{\|u\|_H\leq A}\left(
 |\cE_s(u)-\cE_0(u)|
 +\sum_{j=1}^2\|\dd_u^j\cE_s-\dd_u^j\cE_0\|_{\mathrm{op}}
 \right)\longrightarrow0\quad\text{as }s\downarrow0,
\end{equation}
where $|\cdot|$ is the Euclidean norm and $\|\cdot\|_{\mathrm{op}}$ denotes
the operator norm of linear or bilinear maps.
\end{lemma}

\begin{proof}
The coefficient of $\partial_{z_j}$ in $\widehat X_i$ is a polynomial
of weighted degree $w_j-1$. It therefore depends only on coordinates of
weight strictly less than $w_j$.  In particular, if $\gamma_u^0$ is a horizontal trajectory for the nilpotent approximation with control $u \in H$, for the coordinates $z_1,\dots,z_{n_1}$ of weight one along such curve it holds $\dot{z}_j = \sum_{i=1}^m u_i L_{ij}$, for a constant $m\times n_1$ matrix $L$. Therefore, these coordinates are bounded by a constant times $\|u\|_{L^1}\leq \|u\|_H\leq A$.
Inductively, once all lower-weight coordinates are bounded, the
coefficients in the equations for the next weight are bounded as well;
integrating them against $u$ gives a bound depending only on $A$.
Every trajectory of the nilpotent approximation with $\|u\|_H\leq A$
therefore exists on $[0,1]$ and lies in a compact set $Q_A$ depending only
on $A$.

Choose a larger compact set $\Omega_A$ with
$Q_A\subset\operatorname{int}\Omega_A$, and choose $s_A>0$ so that a
compact neighborhood of $\Omega_A$ lies in $U_s$ for every
$0\leq s\leq s_A$.
On $\Omega_A$, the fields $X_i^s$ and their first derivatives are
uniformly bounded.  Put
$e_s:=\max_i\|X_i^s-X_i^0\|_{C^0(\Omega_A)}$.
Up to the first exit time of $\gamma_u^s$ from $\Omega_A$, Gronwall's inequality gives
\begin{equation}
 \sup_{\tau}|\gamma_u^s(\tau)-\gamma_u^0(\tau)|
 \leq C A e_s\exp(CA).
\end{equation}
The constant $C$ is independent of $s$ and of controls with $\|u\|_H\leq A$.
Since $e_s\to0$, after decreasing $s_A$ the right-hand side is smaller than
$\operatorname{dist}(Q_A,\partial\Omega_A)$ for every $s\in[0,s_A]$.
Hence the trajectory cannot leave $\Omega_A$, and the standard continuation
argument gives existence on $[0,1]$.
Smooth dependence of the fields on $s$ gives
$\max_i\|X_i^s-X_i^t\|_{C^0(\Omega_A)}\leq C|s-t|$ for $s,t\in[0,s_A]$;
the same Gronwall estimate proves \eqref{eq:ode-parameter-bound}.

Applying the above argument with bound $A+1$ and decreasing
$s_A$ accordingly, all endpoint maps $\cE_s$, $s\in[0,s_A]$, are defined on
$\{u\in H:\|u\|_H<A+1\}$, a common open neighborhood of the closed ball
$\|u\|_H\leq A$.
By Lemma~\ref{lem:blowup}(i), the fields $X_i^s$ converge to $X_i^0$ in
$C^2$ on $\Omega_{A+1}$; in particular, their derivatives up to order two
are uniformly bounded there.
Fix $u$ with $\|u\|_H\leq A$.
Taking $t=0$ in \eqref{eq:ode-parameter-bound} gives
\begin{equation}
 \|\gamma_u^s-\gamma_u^0\|_{C^0}\leq C_A s
\end{equation}
uniformly in $u$.
For $v,w\in H$, set $V_v^s:=\dd_u\gamma_u^s[v]$ and
$W_{v,w}^s:=\dd_u^2\gamma_u^s[v,w]$.
Differentiating the integral equation corresponding to \eqref{eq:control-system} with respect to the control gives
\begin{align}
 \dot V_v^s
 &=\sum_i u_i\,\dd_{\gamma_u^s}X_i^s(V_v^s)
   +\sum_i v_iX_i^s(\gamma_u^s),\\
 \dot W_{v,w}^s
 &=\sum_i u_i\,\dd_{\gamma_u^s}X_i^s(W_{v,w}^s)
   +\sum_i u_i\,\dd_{\gamma_u^s}^2X_i^s[V_v^s,V_w^s]
 +\sum_i v_i\,\dd_{\gamma_u^s}X_i^s(V_w^s)
   +\sum_i w_i\,\dd_{\gamma_u^s}X_i^s(V_v^s),
\end{align}
with $V_v^s(0)=W_{v,w}^s(0)=0$. Uniform bounds on $\Omega_{A+1}$ and
Gronwall's inequality give
\begin{equation}
 \|V_v^s\|_{C^0}\leq C_A\|v\|_H,
 \qquad
 \|W_{v,w}^s\|_{C^0}\leq C_A\|v\|_H\|w\|_H.
\end{equation}
Considering the differences $V_v^s-V_v^0$ and $W_{v,w}^s-W_{v,w}^0$, using the $C^2$ convergence of the fields on $\Omega_{A+1}$,
the uniform convergence $\gamma_u^s\to\gamma_u^0$, and the bounds above, Gronwall's inequality gives
\begin{equation}
 \|V_v^s-V_v^0\|_{C^0}\leq\varepsilon_s\|v\|_H,\qquad  \|W_{v,w}^s-W_{v,w}^0\|_{C^0}
 \leq\varepsilon_s\|v\|_H\|w\|_H.
\end{equation}
Here $\varepsilon_s\to0$ as $s\downarrow 0$ is independent
of $u,v,w$ with $\|u\|_H\leq A$. Since
$\cE_s(u)=\gamma_u^s(1)$, $\dd_u\cE_s[v]=V_v^s(1)$ and
$\dd_u^2\cE_s[v,w]=W_{v,w}^s(1)$, taking the supremum over unit vectors
$v,w\in H$ proves the asserted uniform $C^2$ convergence in operator norm.
\end{proof}

For a fixed sub-Riemannian structure, Lipschitz regularity of the distance is established in
\cite[Theorem~5.7]{AgrachevLee} and \cite[Theorem~12.11]{ABB}, under Morse index assumptions.
Here an assumption only on the limiting nilpotent structure yields
estimates uniform along the approximating family, together with
Lipschitz dependence on the approximation parameter. 
It is worth noting that the uniform estimates below do not follow directly from the fixed-structure results by a compactness argument. Indeed, although the endpoint maps converge in $C^2$, their intrinsic Hessians are restricted to $\ker \dd_u\mathcal \cE_s$, and these kernels need not vary continuously, since the rank of $\dd_u\mathcal \cE_s$ may jump at $s=0$. The argument in Step $4$ of the proof overcomes this difficulty.

\begin{proposition}[Uniform Lipschitz estimates]
\label{prop:uniform-stability}
Assume the hypothesis of Theorem~\ref{thm:main}.  For every compact set
$K\subset\R^n\setminus\{0\}$ there exist $s_K>0$ and $L_K>0$ such that,
for all $s,t\in[0,s_K]$,
\begin{align}
 |d_s(0,y)-d_s(0,z)|&\leq L_K|y-z|,
 &&y,z\in K, \label{eq:lip1}\\
 |d_s(0,z)-d_t(0,z)|&\leq L_K|s-t|,
 &&z\in K. \label{eq:lip2}
\end{align}
\end{proposition}

\begin{proof}
Choose a compact neighborhood $K^+$ of $K$ with
$K^+\Subset\R^n\setminus\{0\}$. We first control minimizers and
subdifferentials with endpoints in $K^+$.

\emph{1. Continuity of distances.}
By Lemma~\ref{lem:blowup}(ii), we may choose $s_0>0$ and $A>1$ such that
\begin{equation}
 K^+\subset U_s,\qquad d_s(0,z)\leq A-1,\qquad 0\leq s\leq s_0,\quad z\in K^+.
\end{equation}
We decrease $s_0$ so that Lemma~\ref{lem:compact-containment} applies with bound on the controls equal to $A$ and $s_0(A-1)<r_{\mathrm{loc}}$. For $0<s\leq s_0$ and $z\in K^+$, we have
$s\,d_s(0,z)\leq s_0(A-1)<r_{\mathrm{loc}}$, so
Lemma~\ref{lem:blowup}(iii) gives
\begin{equation}
d_s(0,z)=\frac1s d\bigl(q_0,x^{-1}(\delta_s z)\bigr),
\qquad 0<s\leq s_0,\quad z\in K^+.
\end{equation}
This formula gives continuity in $(s,z)$ for $s>0$, while continuity at
$s=0$ follows from Lemma~\ref{lem:blowup}(ii) and continuity of $d_0$.
Thus $(s,z)\mapsto d_s(0,z)$ is continuous on $[0,s_0]\times K^+$.

\emph{2. Existence of minimizers.}
Fix $(s,z)\in[0,s_0]\times K^+$ and take a sequence $(u_k)_k$ such that
$\cE_s(u_k)=z$, $\cJ(u_k)\to f_s(z)$.
Since $2f_s(z)=d_s(0,z)^2\leq(A-1)^2$, we may assume
$\|u_k\|_H\leq A$.
By Lemma~\ref{lem:compact-containment}, the trajectories $\gamma^s_{u_k}$
lie in a common compact set, and are equicontinuous (actually, equi-Hölder with Hölder exponent $1/2$). After extraction, $u_k\rightharpoonup u$
in $H$ and $\gamma_{u_k}^s\to\gamma$ uniformly. Since
$X_i^s(\gamma_{u_k}^s)\to X_i^s(\gamma)$ uniformly, weak convergence of the
controls allows us to pass to the limit in the integral equations, giving
\begin{equation}
 \gamma_{u_k}^s(t)=\int_0^t\sum_{i=1}^m
 u_{k,i}(\tau)X_i^s(\gamma_{u_k}^s(\tau))\,\dd\tau \longrightarrow \int_0^t\sum_{i=1}^m
 u_i(\tau)X_i^s(\gamma(\tau))\,\dd\tau.
\end{equation}
Thus $\gamma=\gamma_u^s$ and $\cE_s(u)=z$.
Weak lower semicontinuity of the $H$ norm yields
$\cJ(u)\leq\liminf_k\cJ(u_k)=f_s(z)$, so $u$ is energy-minimizing for the $s$-structure.

\emph{3. Compactness of minimizers.} Define
\begin{equation}
 \mathfrak M_{s_0}:=
 \bigl\{(s,u)\in[0,s_0]\times H:\cE_s(u)\in K^+,\
 u\text{ is energy-minimizing for the }s\text{-structure}\bigr\}.
\end{equation}
We claim that $\mathfrak M_{s_0}$ is compact in $[0,s_0]\times H$,
with the strong topology on $H$.
Take a sequence $(s_j,u_j)\in\mathfrak M_{s_0}$ and set
$z_j:=\cE_{s_j}(u_j)$. After extraction,
$s_j\to s\in[0,s_0]$ and $z_j\to z\in K^+$.
Since $\|u_j\|_H=d_{s_j}(0,z_j)\leq A-1$,
Lemma~\ref{lem:compact-containment} and the argument of Step~2 give,
after extraction, $u_j\rightharpoonup u$ in $H$ with uniformly
convergent trajectories.
Using also the smooth dependence of the fields on $s$ from
Lemma~\ref{lem:blowup}(i), taking the limit in the integral
equations as above gives $\cE_s(u)=z$.
By the continuity established in Step~1,
\begin{equation}
 d_s(0,z)\leq\|u\|_H
 \leq\liminf_j\|u_j\|_H
 =\lim_jd_{s_j}(0,z_j)=d_s(0,z).
\end{equation}
Thus $u$ is energy-minimizing for the $s$-structure and $\|u_j\|_H\to\|u\|_H$.
Weak convergence and convergence of the norms in $H$ imply
$u_j\to u$ strongly, proving the claim.

\emph{4. Subdifferential bounds and first Lipschitz estimates.}
We claim that, after reducing $s_0>0$, there is $M>0$ such that
\begin{equation}\label{eq:uniform-proximal-bound}
 |p|\leq M\qquad
 0\leq s\leq s_0,\ z\in K^+,\ p\in\partial^-_P f_s(z).
\end{equation}
Otherwise, for all $j\in \mathbb{N}$, we could choose
\begin{equation*}
 0\leq s_j\leq\min\{s_0,1/j\},\qquad
 z_j\in K^+,\qquad p_j\in\partial^-_P f_{s_j}(z_j),\qquad |p_j|>j.
\end{equation*}
Let $u_j$ be energy-minimizing controls with $\cE_{s_j}(u_j)=z_j$.
Since $s_j\to0$, Step~3 gives, after extraction, $u_j\to u$ strongly
in $H$ and $z_j\to z\in K^+$, where $u$ is energy-minimizing for the
limiting nilpotent structure and $\cE_0(u)=z\neq0$.

Let $\varphi_j$ be a $C^2$ support function from below for $f_{s_j}$
at $z_j$, with $\dd_{z_j}\varphi_j=p_j$. In a neighborhood of $u_j$ it holds the inequality
$f_{s_j}(\cE_{s_j}(v))\leq\cJ(v)$, with equality for $v=u_j$. This implies that
$\cJ-\varphi_j\circ\cE_{s_j}$ has a local minimum at $u_j$.
The first- and second-order necessary conditions give
\begin{align*}
 p_j\circ\dd_{u_j}\cE_{s_j}&=\dd_{u_j}\cJ,\\
 \|v\|_H^2-p_j\bigl(\dd_{u_j}^2\cE_{s_j}[v,v]\bigr)&\geq0,
 \qquad v\in\ker\dd_{u_j}\cE_{s_j}.
\end{align*}
Set $\lambda_j:=-p_j/|p_j|$ and extract a subsequence with
$\lambda_j\to\lambda$, $|\lambda|=1$. Lemma~\ref{lem:compact-containment}
and the strong convergence $u_j\to u$ imply
\begin{equation*}
 \dd_{u_j}^{\ell}\cE_{s_j}\longrightarrow\dd_u^{\ell}\cE_0
 \quad\text{in operator norm},\qquad \ell=1,2.
\end{equation*}
Since $(u_j)$ is bounded,
\begin{equation*}
 \lambda_j\circ\dd_{u_j}\cE_{s_j}
 =-\frac{\dd_{u_j}\cJ}{|p_j|}\longrightarrow0\quad\text{in }H^*,
\end{equation*}
so $\lambda\circ\dd_u\cE_0=0$. Therefore, $\lambda$ is an abnormal
multiplier for $u$. Moreover,
\begin{equation}\label{eq:normalized-second-variation}
 \lambda_j\bigl(\dd_{u_j}^2\cE_{s_j}[v,v]\bigr)
 \geq-\frac{\|v\|_H^2}{|p_j|},\qquad  v\in\ker\dd_{u_j}\cE_{s_j}.
\end{equation}

We show that
\begin{equation}\label{eq:limiting-index-bound}
 \operatorname{ind}_-\bigl(\lambda\circ\operatorname{Hess}_u\cE_0\bigr)
 \leq k,\qquad k:=\operatorname{codim}\im\dd_u\cE_0.
\end{equation}
Otherwise, there would be a $(k+1)$-dimensional subspace
$W\subset\ker\dd_u\cE_0$ on which
$q_\lambda(v):=\lambda(\dd_u^2\cE_0[v,v])$ is negative definite.
To use \eqref{eq:normalized-second-variation}, we need vectors in
$\ker\dd_{u_j}\cE_{s_j}$ close to unit vectors of $W$. We give the idea of the proof. We first modify \emph{any} given $w \in W$ introductin a linear correction in the space
$(\ker\dd_u\cE_0)^\perp$ to cancel the component of $\dd_{u_j}\cE_{s_j}(w)$ along
$\im\dd_u\cE_0$. Then, a dimension argument using $\dim W=k+1$ will select a suitable $w$ so that also the remaining $(\im\dd_u\cE_0)^\perp$-component vanishes.

Decompose $\R^n = \im\dd_u\cE_0 \oplus (\im\dd_u\cE_0)^\perp$, and let $P,Q$ be the orthogonal projections onto $\im\dd_u\cE_0,(\im\dd_u\cE_0)^\perp$, respectively.
The restriction $\dd_u\cE_0|_{(\ker\dd_u\cE_0)^\perp}$ is an
isomorphism onto $\im\dd_u\cE_0$, and
\begin{equation*}
 P\dd_{u_j}\cE_{s_j}|_{(\ker\dd_u\cE_0)^\perp}
 \longrightarrow\dd_u\cE_0|_{(\ker\dd_u\cE_0)^\perp}
 \quad\text{in operator norm}.
\end{equation*}
Thus, for large $j$, the inverses
\begin{equation*}
 B_j:=\left(P\dd_{u_j}\cE_{s_j}|_{(\ker\dd_u\cE_0)^\perp}\right)^{-1}
 :\im\dd_u\cE_0\longrightarrow(\ker\dd_u\cE_0)^\perp
\end{equation*}
exist and are uniformly bounded. For $w\in W$, the solution of $P\dd_{u_j}\cE_{s_j}(w+a)=0$ for
$a\in(\ker\dd_u\cE_0)^\perp$ is given by
\begin{equation*}
 a_j(w):=-B_jP\dd_{u_j}\cE_{s_j}(w).
\end{equation*}
Since $P\dd_{u_j}\cE_{s_j}(w)
=P(\dd_{u_j}\cE_{s_j}-\dd_u\cE_0)(w)$ on $W$, the uniform bound
on $B_j$ gives $\|a_j\|_{\mathrm{op}}\to0$ as linear maps from $W$
to $H$.
Set $T_j(w):=Q\dd_{u_j}\cE_{s_j}(w+a_j(w))$.
As $\dim W=k+1>k=\dim (\im\dd_u\cE_0)^\perp$, the linear map $T_j:W\to (\im\dd_u\cE_0)^\perp$
has a nontrivial kernel. Then, we can choose unit vectors $w_j\in\ker T_j$ and set
$v_j:=w_j+a_j(w_j)$. By construction 
\[
P\dd_{u_j}\cE_{s_j}(v_j)= Q\dd_{u_j}\cE_{s_j}(v_j) =0,
\]
so $v_j\in\ker\dd_{u_j}\cE_{s_j}$.
After extraction, $w_j\to w\in W$ with $\|w\|_H=1$, and $v_j\to w$.
Thus
\begin{equation*}
 \lambda_j\bigl(\dd_{u_j}^2\cE_{s_j}[v_j,v_j]\bigr)
 \longrightarrow q_\lambda(w)<0,
\end{equation*}
contradicting \eqref{eq:normalized-second-variation}, whose right-hand
side tends to zero. This proves \eqref{eq:limiting-index-bound}. The
existence of this multiplier implies
$\operatorname{ind}(\gamma_u)\leq k=\operatorname{corank}(\gamma_u)$ for the limit trajectory,
contradicting the hypothesis of Theorem~\ref{thm:main}. Hence \eqref{eq:uniform-proximal-bound} holds.

Lemma~\ref{lem:proximal-approximation} now gives the same bound $M$ for
all Fr\'echet subdifferentials on $\operatorname{int}K^+$: the approximating
proximal points eventually lie in $K^+$, where
\eqref{eq:uniform-proximal-bound} applies.
Choose a compact neighborhood $K'$ of $K$ contained in
$\operatorname{int}K^+$ (this is needed in Step~5 where we will perturb the endpoint). Choose $r>0$ so that
$B(z,2r)\subset\operatorname{int}K^+$ for every $z\in K'$.
Proposition~\ref{prop:subdifferential-lipschitz} on these balls gives
$|f_s(y)-f_s(z)|\leq M|y-z|$ for $y,z\in K'$ with $|y-z|\leq r$,
uniformly in $s\in[0,s_0]$. For $|y-z|>r$, use
$0\leq f_s\leq\tfrac12(A-1)^2$ from Step~1. With
$C_1:=\max\{M,(A-1)^2/(2r)\}$, we obtain
\begin{equation}\label{eq:f-spatial-lipschitz}
 |f_s(y)-f_s(z)|\leq C_1|y-z|
 \qquad 0\leq s\leq s_0,\ y,z\in K'.
\end{equation}

\emph{5. Lipschitz dependence on the parameter.} To prove Lipschitz continuity for $[0,s_0]\ni s\mapsto f_s(z)$, uniformly in $z$, let $u_s$ be energy-minimizing from $0$ to
$z\in K$. Since $\|u_s\|_H=d_s(0,z)\leq A-1$, Lemma~\ref{lem:compact-containment}
implies that $u_s\in\mathcal U_t$ for every $t\in[0,s_0]$. Equation
\eqref{eq:ode-parameter-bound} gives
\begin{equation}
 z':=\cE_t(u_s),\qquad |z'-z|\leq C_2|t-s|,
\end{equation}
with $C_2$ independent of the minimizer. Reduce $s_0$ so that
\begin{equation}
 C_2s_0<
 \operatorname{dist}\bigl(K,\R^n\setminus\operatorname{int}K'\bigr).
\end{equation}
Then $z'\in K'$, and it holds
$f_t(z')\leq\cJ(u_s)=f_s(z)$. By \eqref{eq:f-spatial-lipschitz},
\begin{equation}
 f_t(z)\leq f_t(z')+C_1|z-z'|
 \leq f_s(z)+C_1C_2|t-s|.
\end{equation}
Interchanging $s$ and $t$ proves the second Lipschitz estimate for $f_s$.

\emph{6. Lipschitz estimates for the distance.} Deducing \eqref{eq:lip1}--\eqref{eq:lip2} is now straightforward. Since $(s,z)\mapsto d_s(0,z)$ is continuous on
$[0,s_0]\times K^+$ and $K^+\Subset\R^n\setminus\{0\}$, there is $c_1>0$ such that
\begin{equation}
d_s(0,z)\geq c_1 > 0,\qquad  0\leq s\leq s_0,\quad z\in K^+.
\end{equation}
Thus, for $s,t\in[0,s_0]$ and $y,z\in K$, it holds
\begin{equation*}
    |d_s(0,y)-d_t(0,z)| =\frac{2|f_s(y)-f_t(z)|}{d_s(0,y)+d_t(0,z)} \leq \frac{|f_s(y)-f_t(z)|}{c_1}.
\end{equation*}
Then, \eqref{eq:lip1}--\eqref{eq:lip2} for $d_s$ follow from the corresponding ones for $f_s$ proved above, adjusting the constants appropriately. Set $s_K:=s_0$ and $L_K:=\max\{C_1,C_1C_2\}/c_1$ to conclude.
\end{proof}

The same argument applies for both variable endpoints. We record the result, of independent interest, but since it is not used for the proof of Theorem \ref{thm:main} we only sketch its proof.
\begin{proposition}[Uniform Lipschitz estimates II]\label{prop:two-point-lipschitz}
Assume that every nonconstant abnormal length-minimizing trajectory of
the nilpotent approximation, with arbitrary initial point, has Morse
index strictly larger than its corank. Let
$\Delta:=\{(x,x):x\in\R^n\}$. Then, for every compact set
$\mathcal K\Subset(\R^n\times\R^n)\setminus\Delta$, there exist
$s_{\mathcal K}>0$ and $L_{\mathcal K}>0$ such that
\begin{equation*}
 |d_s(x,y)-d_t(x',y')|
 \leq L_{\mathcal K}\bigl(|x-x'|+|y-y'|+|s-t|\bigr)
\end{equation*}
for all $(x,y),(x',y')\in\mathcal K$ and
$s,t\in[0,s_{\mathcal K}]$.
    \end{proposition}

\begin{proof}
For initial points $x$ in a fixed compact set, the argument and estimates in Lemma~\ref{lem:compact-containment}
remain uniform in $x$. We also obtain uniform
$C^2$ convergence of the joint endpoint maps
$\mathcal F_s(x,u):=(x,\cE_{s,x}(u))$ on compact sets of initial points
and bounded sets of controls.

Write $A:=\dd_u\cE_{0,x}$ and $B:=\dd_x\cE_{0,x}(u)$. Then
$\dd_{(x,u)}\mathcal F_0(\xi,v)=(\xi,B\xi+Av)$, so its kernel is
$\{0\}\times\ker A$ and its corank equals that of $A$.
Its abnormal multipliers are exactly
$(-\beta\circ B,\beta)$, where $0\neq\beta\in(\im A)^\perp$.
On the kernel, the corresponding Hessian is
$v\mapsto\beta(\dd_u^2\cE_{0,x}[v,v])$.
Thus the index hypothesis is unchanged for the joint endpoint map.

We can then apply the argument of the proof of Proposition~\ref{prop:uniform-stability} to
$f_s(x,y):=\tfrac12d_s(x,y)^2$, for $x,y$ in compact neighborhoods of
$\mathcal K$ disjoint from $\Delta$ and the cost
$(x,u)\mapsto\cJ(u)$. 
\end{proof}

\section{Proof of the main theorems}\label{sec:radial}

For the assertion about open balls in Theorem \ref{thm:main}, we use the following lemma. It is an anisotropic version of the well-known fact that any open
star-shaped subset of $\R^n$ is diffeomorphic to $\R^n$, and is proved  in Appendix~\ref{app:smooth-straightening}.

\begin{lemma}
\label{lem:smooth-straightening}
Let $D\subset\R^n$ be an open, dilation-star-shaped neighborhood of the origin, namely
\begin{equation}
z\in D,\quad 0\le t\le1
\quad\Longrightarrow\quad
\delta_t z\in D.
\end{equation}
Then $D$ is diffeomorphic to $\R^n$.
\end{lemma}

\subsection*{Proof of Theorem~\ref{thm:main}}
Fix $\varepsilon>0$, that we will reduce throughout the proof.
For $\varepsilon\leq r_{\mathrm{loc}}$, Lemma~\ref{lem:localization} and
\eqref{eq:local-distance-rescaling} show that the diffeomorphism
\begin{equation}
 \Phi_\varepsilon:U\longrightarrow U_\varepsilon,
 \qquad \Phi_\varepsilon(q)=\delta_{1/\varepsilon}x(q),
\end{equation}
maps the open and closed balls of radius $\varepsilon$ for the original
sub-Riemannian structure onto the corresponding unit balls for
$d_\varepsilon$ on $U_\varepsilon$:
\begin{equation}
 D_\varepsilon:=\{z\in U_\varepsilon:d_\varepsilon(0,z)<1\},
 \qquad
 \overline D_\varepsilon:=\{z\in U_\varepsilon:d_\varepsilon(0,z)\leq1\},
\end{equation}
and the sphere $S_{q_0}(\varepsilon)$ onto $\{d_\varepsilon(0,\cdot)=1\}$.
We study these unit balls and spheres.

Along each ray $r\mapsto\delta_r\theta$, $\theta\in S^{n-1}$, we will
show that the distance $d_\varepsilon$ grows strictly from $0$ to a value greater than
$1$. Thus the unit $d_\varepsilon$-sphere will be a continuous radial graph, equal to
$\partial D_\varepsilon$.
Every nonzero point lies on exactly one such ray.

Proposition~\ref{prop:uniform-stability}, applied to the compact set $S^{n-1}=\{z\in \R^n : |z|=1\}$, gives
$s_*>0$ and $L>0$ such that $S^{n-1}\subset U_s$ for $0\leq s\leq s_*$ and
\begin{equation}\label{eq:sphere-parameter-lipschitz}
 |d_s(0,\theta)-d_t(0,\theta)|\leq L|s-t|,
 \qquad \theta\in S^{n-1},\ 0\leq s,t\leq s_*.
\end{equation}
By Lemma~\ref{lem:blowup}(ii), after decreasing $s_*$, we also have
\begin{equation}\label{eq:sphere-lower-bound}
 d_s(0,\theta)\geq c:=\tfrac12\min_{S^{n-1}}d_0(0,\cdot)>0,
 \qquad \theta\in S^{n-1},\ 0\leq s\leq s_*.
\end{equation}

To enclose the unit $d_\varepsilon$-ball, we seek a reference set
$\delta_R(S^{n-1})$, diffeomorphic to $S^{n-1}$, on which $d_\varepsilon>1$. We choose $R>1/c$.
With $R$ fixed, reduce $\varepsilon$ so that $\varepsilon R\leq s_*$.
For $0<r\leq R$, we then have $S^{n-1}\subset U_{\varepsilon r}$;
since $\delta_rU_{\varepsilon r}=U_\varepsilon$, the entire region
\begin{equation}
 \delta_R\overline B^n
 =\{\delta_r\theta:0\leq r\leq R,\ \theta\in S^{n-1}\}
 \subset U_\varepsilon
\end{equation}
lies in the domain of $d_\varepsilon$. Moreover,
\eqref{eq:two-parameter-scaling} and \eqref{eq:sphere-lower-bound} give
\begin{equation}
 d_\varepsilon(0,\delta_R\theta)
 =R\,d_{\varepsilon R}(0,\theta)\geq Rc>1.
\end{equation}
Every horizontal curve $\gamma$ for the $\varepsilon$-structure in $U_\varepsilon$ from $0$ to a point
outside $\delta_R B^n$ meets $\delta_R S^{n-1}$. Its length is therefore at least $Rc$.
Taking infima gives $\overline D_\varepsilon\subset\delta_R B^n$.

For $0<r_1<r_2\leq R$, \eqref{eq:two-parameter-scaling} and
\eqref{eq:sphere-parameter-lipschitz}--\eqref{eq:sphere-lower-bound} yield
\begin{equation}\label{eq:monotonicity}
\begin{aligned}
 d_\varepsilon(0,\delta_{r_2}\theta)
 -d_\varepsilon(0,\delta_{r_1}\theta)
 & =(r_2-r_1)d_{\varepsilon r_2}(0,\theta)
 +r_1\bigl(d_{\varepsilon r_2}(0,\theta)
          -d_{\varepsilon r_1}(0,\theta)\bigr)\\
 &\geq(r_2-r_1)(c-LR\varepsilon).
\end{aligned}
\end{equation}
Thus requiring $LR\varepsilon<c$ makes
$r\mapsto d_\varepsilon(0,\delta_r\theta)$ strictly increasing on
$[0,R]$, for all $\theta$.
By continuity and its values at $0$ and $R$, there is a unique
$r_\varepsilon(\theta)\in(0,R)$ satisfying
\begin{equation}\label{eq:radial-root}
 d_\varepsilon\bigl(0,\delta_{r_\varepsilon(\theta)}\theta\bigr)=1.
\end{equation}
The function $r_\varepsilon :S^{n-1}\to (0,R)$ is continuous: if $\theta_j\to\theta$,
every limit point of $r_\varepsilon(\theta_j)$ in $[0,R]$ satisfies
\eqref{eq:radial-root}, hence it is equal to $r_\varepsilon(\theta)$.
Strict monotonicity now gives
\begin{align}
 D_\varepsilon
 &=\{\delta_r\theta:\theta\in S^{n-1},\
                         0\leq r<r_\varepsilon(\theta)\},
                         \label{eq:open-radial-description}\\
 \overline D_\varepsilon
 &=\{\delta_r\theta:\theta\in S^{n-1},\
                         0\leq r\leq r_\varepsilon(\theta)\},\notag\\
 \partial D_\varepsilon
 &=\{\delta_{r_\varepsilon(\theta)}\theta:\theta\in S^{n-1}\}
   =\{d_\varepsilon(0,\cdot)=1\}.\notag
\end{align}
In particular, $\overline D_\varepsilon$ is the closure of
$D_\varepsilon$.
The map $H_\varepsilon:\R^n\to\R^n$ defined by
\begin{equation}
 H_\varepsilon(z) := \begin{cases}
0, & \text{if $z=0$},\\
\delta_{t r_\varepsilon(\theta)}\theta, & \text{if $z= t\theta\neq 0$}.    
\end{cases}
\end{equation}
is a continuous bijection from $\overline B^n$ onto
$\overline D_\varepsilon$ (continuity at $0$ follows from the fact that
$r_\varepsilon \in (0,R)$ is bounded). By compactness, $H_\varepsilon$ is a homeomorphism of pairs
$(\overline B^n,S^{n-1})\to
(\overline D_\varepsilon,\partial D_\varepsilon)$.
Composing with $\Phi_\varepsilon^{-1}$ yields the required homeomorphism
with $(\overline B_{q_0}(\varepsilon),S_{q_0}(\varepsilon))$.

Finally, by definition, $D_\varepsilon$ is
dilation-star-shaped, so Lemma~\ref{lem:smooth-straightening} gives a diffeomorphism
$D_\varepsilon\simeq\R^n\simeq B^n$.
Composing with $\Phi_\varepsilon^{-1}$ we obtain a diffeomorphism with
$B_{q_0}(\varepsilon)$.
The finitely many upper bounds on $\varepsilon$ can be satisfied by choosing $0<\varepsilon\leq \varepsilon_*$ sufficiently small. \hfill \qed

\subsection*{Proof of Theorem~\ref{thm:lipschitz-spheres}}
Use the notation of the previous proof and reduce $\varepsilon$ so that
$c-LR\varepsilon\geq c/2$. By \eqref{eq:two-parameter-scaling},
$1=r_\varepsilon(\theta)d_{\varepsilon r_\varepsilon(\theta)}(0,\theta)$;
the uniform upper bound in \eqref{eq:sphere-parameter-lipschitz} therefore
gives $r_\varepsilon(\theta)\geq a>0$, independently of $\varepsilon$.
Set $K:=\{\delta_r\theta:\theta\in S^{n-1},\ a\leq r\leq R\}
\Subset\R^n\setminus\{0\}$ and further require $\varepsilon\leq s_K$.
The radial monotonicity estimate \eqref{eq:monotonicity} and
Proposition~\ref{prop:uniform-stability} give
\begin{align}
 \frac c2\,|r_\varepsilon(\theta)-r_\varepsilon(\eta)|
& \leq\bigl|d_\varepsilon(0,\delta_{r_\varepsilon(\theta)}\theta)
            -d_\varepsilon(0,\delta_{r_\varepsilon(\eta)}\theta)\bigr|\\
            & =\bigl|d_\varepsilon(0,\delta_{r_\varepsilon(\eta)}\eta)
            -d_\varepsilon(0,\delta_{r_\varepsilon(\eta)}\theta)\bigr|\leq C|\theta-\eta|, \qquad \theta,\eta\in S^{n-1},
\end{align}
where $C=L_K\max_{a\leq r\leq R}\|\delta_r\|$ does not depend on
$\varepsilon$.

Hence $r_\varepsilon$ and
$\theta\mapsto\delta_{r_\varepsilon(\theta)}\theta$ are uniformly
Lipschitz. The inverse parametrizations are restrictions of the smooth
projection along dilation orbits onto $S^{n-1}$, which is Lipschitz on
$K$, so the parametrizations are uniformly bi-Lipschitz.
In the smooth coordinates $(\theta,r)\mapsto\delta_r\theta$, the spheres
are locally Lipschitz graphs $r=r_\varepsilon(\theta)$, hence embedded
Lipschitz hypersurfaces. The latter property passes to the original
spheres through the diffeomorphism $\Phi_\varepsilon^{-1}$. \hfill \qed

\appendix

\section{Some nonsmooth analysis facts}
\label{app:nonsmooth}

We give elementary proofs of the two standard facts stated
Section~\ref{sec:subdifferentials}.

\subsection*{Proof of Lemma~\ref{lem:proximal-approximation}}
Choose a $C^1$ support function $\varphi$ at $x$ with
$\dd_x\varphi=p$, and a closed Euclidean ball $D:=\overline B(x,r)\subset O$
on a neighborhood of which $\varphi$ is defined and $f\geq\varphi$.
Replacing $\varphi(y)$ by $\varphi(y)-|y-x|^2$, if necessary, we may
assume that $x$ is the unique contact point between $f$ and $\varphi$
on $D$, without changing $\dd_x\varphi$.

By mollification, choose smooth $\varphi_j$ converging to $\varphi$
in $C^1(D)$. Let $y_j$ minimize
\begin{equation*}
 f(y)-\varphi_j(y),\qquad y\in D.
\end{equation*}
These functions converge uniformly to $f-\varphi$, whose unique
minimum on $D$ is at $x$. Compactness therefore gives $y_j\to x$,
so $y_j\in\operatorname{int}D$ for large $j$. Choose $c_j\in \R$ so that
\begin{equation*}
 \psi_j(y):=\varphi_j(y)+c_j,
 \qquad \psi_j(y_j)=f(y_j).
\end{equation*}
By construction, $\psi_j\leq f$ on $D$, so
$q_j:=\dd_{y_j}\psi_j\in\partial^-_P f(y_j)$.
The $C^1$ convergence gives $q_j\to\dd_x\varphi=p$, so
$(y_j,q_j)\to(x,p)$, as required. \hfill \qed

\subsection*{Proof of Proposition~\ref{prop:subdifferential-lipschitz}}
If $f$ is $L$-Lipschitz and $\varphi$ is a $C^1$ support function
at $x$, then, for every unit vector $e$ and small $t>0$,
\begin{equation*}
 \varphi(x+te)-\varphi(x)\leq f(x+te)-f(x)\leq Lt.
\end{equation*}
Dividing by $t$ and taking the supremum over $e$ gives
$|\dd_x\varphi|\leq L$.

Conversely, suppose that all Fr\'echet subdifferentials on $C$ have norm
at most $L$. Fix distinct $a,b\in C$. By convexity, the segment
$S:=[a,b]$ lies in $C$; choose $\rho>0$ so that the compact tube
$D:=\{y:\operatorname{dist}(y,S)\leq\rho\}$ is contained in $C$.
For $\varepsilon>0$ and $x\in S$, set
\begin{equation*}
 f_\varepsilon(x):=
 \min_{y\in D}\left\{f(y)+\frac{|x-y|^2}{2\varepsilon}\right\}.
\end{equation*}
Every minimizer $y_x$ satisfies
\begin{equation*}
 |x-y_x|^2\leq2\varepsilon\left(\max_D f-\min_D f\right),
\end{equation*}
by comparison with $y=x$. Thus, for small $\varepsilon$, all such
minimizers lie in $\operatorname{int}D$, uniformly for $x\in S$.
The smooth function
$y\mapsto f_\varepsilon(x)-|x-y|^2/(2\varepsilon)$ is a support function from below for $f$
 at $y_x$. Using the Euclidean identification of vectors and covectors, we obtain
\begin{equation*}
 p_x:=\frac{x-y_x}{\varepsilon}\in\partial^-_P f(y_x)
       \subseteq\partial^-f(y_x),\qquad |p_x|\leq L.
\end{equation*}
Using $y_x$ as a competitor at $x'\in S$ gives
\begin{equation*}
 f_\varepsilon(x')-f_\varepsilon(x)
 \leq p_x(x'-x)+\frac{|x'-x|^2}{2\varepsilon}
 \leq L|x'-x|+\frac{|x'-x|^2}{2\varepsilon}.
\end{equation*}
Interchanging $x,x'$ and summing over a partition of $S$ into $N$
equal segments yields
\begin{equation*}
 |f_\varepsilon(b)-f_\varepsilon(a)|
 \leq L|b-a|+\frac{|b-a|^2}{2\varepsilon N}.
\end{equation*}
First let $N\to\infty$, and then let $\varepsilon\downarrow0$. By uniform continuity of $f$ on $D$ and the bound on $|x-y_x|$, we obtain that $f_\varepsilon\to f$ uniformly on $S$. This proves
$|f(b)-f(a)|\leq L|b-a|$. \hfill \qed

\section{Smoothness of dilation-star-shaped sets}
\label{app:smooth-straightening}

\subsection*{Proof of Lemma~\ref{lem:smooth-straightening}}
Choose $\psi\in C^\infty(\R^n,[0,\infty))$ with
$\{\psi>0\}=D$, and set $\varphi:=\psi/(1+\psi)$.
Thus $0<\varphi\le1$ on $D$. Define
\[
\sigma(z):=\int_0^1\frac{\dd t}{\varphi(\delta_tz)},
\qquad F(z):=\delta_{\sigma(z)}z.
\]
Since $D$ is dilation-star-shaped, $\sigma$ and $F$ are smooth on $D$.

Every $z\ne0$ is uniquely written as $z=\delta_r\theta$, where
$r>0$ and $\theta\in S^{n-1}$. For each $\theta$, the intersection
of $D$ with this dilation ray is $0\le r<r_D(\theta)$, for some
$r_D(\theta)\in(0,\infty]$. Along the ray,
\[
F(\delta_r\theta)=\delta_{T_\theta(r)}\theta,
\qquad
T_\theta(r):=\int_0^r\frac{\dd u}{\varphi(\delta_u\theta)},
\qquad T_\theta'(r)>0.
\]
If $r_D(\theta)<\infty$, then
$\varphi(\delta_{r_D(\theta)}\theta)=0$. By smoothness we obtain
$\varphi(\delta_r\theta)\le C_\theta(r_D(\theta)-r)$
near $r_D(\theta)$. Hence $T_\theta(r)\to\infty$ as
$r\uparrow r_D(\theta)$. If $r_D(\theta)=\infty$,
the bound $\varphi\le1$ gives $T_\theta(r)\ge r$, with the same conclusion.

Thus $F$ is a bijective map of each dilation ray in $D$ onto the full
corresponding ray in $\R^n$. Since $T_\theta'(r)>0$, $F$ is a local
diffeomorphism away from the origin. At the origin,
\[
DF(0)=\operatorname{diag}\bigl(\sigma(0)^{w_1},\ldots,
\sigma(0)^{w_n}\bigr)
\]
is invertible. Therefore $F:D\to\R^n$ is a diffeomorphism. \hfill \qed

\bibliographystyle{abbrv}
\bibliography{bibliography}

\end{document}